\documentclass[11pt]{article}

\usepackage{latexsym,color,epsfig,graphicx,mathrsfs}
\usepackage{fullpage,sectsty}
\usepackage[titles]{tocloft}
\usepackage[compact]{titlesec}
\titleformat{\subsection}[runin]{\normalfont\bfseries}{\thesubsection}{0.5em}{}
\titlespacing{\section}{0pt}{*4}{*2}
\titlespacing{\subsection}{0pt}{*3}{*1}

\usepackage{amsfonts}
\usepackage{amsmath}
\usepackage{amssymb}
\usepackage{amsthm}
\usepackage{verbatim}
\usepackage{hyperref}
\usepackage{tikz}

\hypersetup{
  colorlinks,
  urlcolor=black,
  linkcolor=black,
  anchorcolor=black,
  citecolor=black,
}

\allowdisplaybreaks

\author{Tom Bohman\thanks{Department of Mathematical Sciences, Carnegie Mellon University, Pittsburgh, PA 15213. Research supported by NSF grant DMS-1100215.} \and Dhruv Mubayi\thanks{Department of Mathematics, Statistics, and Computer Science, University of Illinois at Chicago, Chicago, IL 60607. Research supported by NSF grant DMS-1300138.}
 \and Michael Picollelli \thanks{Department of Mathematics, California State University San Marcos, San Marcos, CA 92096.}}

\date{}

\renewcommand{\phi}{\varphi}
\renewcommand{\l}{\ell}

\newcommand{\ep}{\varepsilon}
\newcommand{\bin}[2]{\binom{#1}{#2}}

\newcommand{\pr}[1]{\mathbb{P}\left[#1\right]}
\newcommand{\ev}[1]{\mathbb{E}\left[#1\right]}

\newcommand{\lp}{\left(}
\newcommand{\rp}{\right)}

\newtheorem{theorem}{Theorem}

\newtheorem{definition}{Definition}
\newtheorem{lemma}[theorem]{Lemma}
\newtheorem{corollary}[theorem]{Corollary}
\newtheorem{claim}{Claim}

\newcommand{\ca}[1]{{\cal #1}}

\newcommand{\im}{i_{\max}}
\newcommand{\tm}{t_{\max}}

\renewcommand{\phi}{\varphi}

\begin{document}

\title{The independent neighborhoods process}

\maketitle

\begin{abstract}

A  triangle $T^{(r)}$ in an $r$-uniform hypergraph
is a set of $r+1$ edges such that $r$ of them share a common $(r-1)$-set of vertices and the last edge contains the remaining vertex from each of the first $r$ edges.  Our main result is that the random greedy triangle-free process on $n$ points terminates in an $r$-uniform hypergraph with independence number $O((n \log n)^{1/r})$.  As a consequence, using recent results on independent sets in hypergraphs, the Ramsey number $r(T^{(r)}, K_s^{(r)})$ has order of magnitude $s^r/\log s$. This answers questions posed in~\cite{BFM, KMV} and generalizes the celebrated results of Ajtai-Koml\'os-Szemer\'edi~\cite{AKS} and
Kim~\cite{K} to hypergraphs.

\end{abstract}

\section{Introduction}
An $r$-uniform hypergraph $H$ ($r$-graph for short) is a collection of $r$-element subsets of a vertex set $V(H)$.
Given $r$-graphs $G$ and $H$, the ramsey number $r(G, H)$ is the minimum $n$ such that every red/blue-edge coloring of the complete $r$-graph $K_n^{(r)}:={[n] \choose r}$  contains a red copy of $G$ or a blue copy of $H$ (often we will write $K_n$ for $K_n^{(r)}$).
Determining these numbers for graphs ($r=2$) is known to be notoriously difficult, indeed the order of magnitude (for fixed $t$) of $r(K_t, K_s)$ is wide open when $t \ge 4$.  The case $t=3$ is one of the celebrated results in graph Ramsey theory:
\begin{equation} \label{r3s} r(K_3,K_s)=\Theta(s^2/\log s).\end{equation}
 The upper bound was proved by Ajtai-Koml\'os-Szemer\'edi~\cite{AKS} as one of the first applications of the semi-random method in combinatorics (simpler proofs now exist due to Shearer~\cite{S83, S91}). The lower bound, due to Kim~\cite{K}, was also achieved by   using the semi-random or nibble method. More recently, the first author~\cite{B} showed that a lower bound for $r(K_3, K_s)$ could also be obtained by the triangle-free process, which is a random greedy algorithm. This settled a question of Spencer on the independence number of the triangle-free process. Still more recently, Bohman-Keevash~\cite{BK2} and Fiz Pontiveros-Griffiths-Morris~\cite{FGM} have analyzed the triangle-free process more carefully and improved the constants obtained so that the gap between the upper and lower bounds for $r(K_3, K_s)$ is now asymptotically a multiplicative factor of 4.

Given the difficulty of these basic questions in graph Ramsey theory, one would expect that the corresponding questions for hypergraphs are hopeless. This is not always the case.  Hypergraphs behave quite differently for asymmetric Ramsey problems, for example, there exist $K_4^{(3)}$-free 3-graphs on $n$ points with independence number of order $\log n$, so $r(K_4^{(3)}, K_s^{(3)})$ is exponential in $s$ unlike the graph case.
Consequently, to obtain $r$-graph results parallel to (\ref{r3s}), one must consider problems $r(G, K_s)$ where $G$ is much sparser than a complete graph. A recent result in this vein  due to Kostochka-Mubayi-Verstra\"ete~\cite{KMV} is that there are positive constants $c_1, c_2$ with
$$\frac{c_1 s^{3/2}}{(\log s)^{3/4}} < r(C_3^{(3)}, K_s^{(3)})< c_2 s^{3/2}$$
where $C_3^{(3)}$ is the loose triangle, comprising 3 edges that have pairwise intersections of size one and have no point in common. The authors in \cite{KMV} conjectured that $r(C_3^{(3)}, K_s^{(3)})=o(s^{3/2})$ and  the order of magnitude remains open. Another result of this type for hypergraphs due to Phelps and R\"odl~\cite{PR} is that $r(P_2^{(3)}, K_s^{(3)})=\Theta(s^2/\log s)$, where $P_t^{(3)}$ is the tight path with $t$ edges.  Recently, the second author and Cooper~\cite{CM} prove that for fixed $t \ge 4$, the behavior of this Ramsey number changes and we have
$r(P_t^{(3)}, K_s^{(3)})=\Theta(s^2)$; the growth rate for $t=3$ remains open. These are the only nontrivial hypergraph results of polynomial Ramsey numbers, and in this paper we add to this list with an extension of (\ref{r3s}).

\begin{definition} An $r$-uniform triangle $T^{(r)}$ is a set of
$r+1$ edges $b_1, \ldots, b_r, a$ with $b_i \cap b_j=R$ for all $i<j$ where $|R|=r-1$ and $a = \cup_i (b_i-R)$. In other words, $r$ of the edges share a common $(r-1)$-set of vertices, and the last edge contains the remaining point in all these previous edges.
\end{definition}

 When $r=2$, then $T^{(2)}=K_3$, so in this sense $T^{(r)}$ is a generalization of a graph triangle.  We may view a $T^{(r)}$-free $r$-graph
as one in which all  neighborhoods are independent sets, where the neighborhood of an $R \in {V(H)\choose r-1}$ is
$\{x: R \cup \{x\} \in H\}$. Frieze and the first two authors~\cite{BFM} proved that for fixed $r \ge 2$, there are positive constants $c_1$ and $c_2$ with
$$c_1\frac{s^r}{(\log s)^{r/(r-1)}}<r(T^{(r)}, K_s^{(r)})< c_2 s^r.$$
They conjectured that the upper bound could be improved to $o(s^r)$ and believed that the log factor in the lower bound could also be improved. Kostochka-Mubayi-Verstra\"ete~\cite{KMV} partially achieved this by improving the upper bound to
$$r(T^{(r)}, K_s^{(r)}) = O(s^r/\log r)$$
and believed that the log factor was optimal.

In this paper we verify this assertion by analyzing the $T^{(r)}$-free (hyper)graph process.
This process begins with an empty hypergraph $G(0)$ on $n$ vertices. Given $G(i-1)$, the hypergraph $G(i)$ is then formed by adding an edge $e_i$ selected uniformly at random from the $r$-sets of vertices which neither form edges of $G(i-1)$ nor create a copy of $T^{(r)}$ in the hypergraph $G(i-1) + e_i$.  The process terminates with a maximal $T^{(r)}$-free graph $G(M)$ with a random number $M$ of edges.  Our main result is the following:
\begin{theorem} \label{main}
For $r \ge 3$ fixed
the $T^{(r)}$-free process on $n$ points produces an $r$-graph with independence
number $ O\left( ( n \log n)^{1/r} \right) $ with high probability.
\end{theorem}
\noindent
This result together with the aformentioned result of
Kostochka-Mubayi-Verstra\"ete give the following generalization of (\ref{r3s})
to hypergraphs.
\begin{corollary}
For fixed $r \ge 3$ there are positive constants $c_1$ and $c_2$ with
$$c_1\frac{s^r}{\log s}<r(T^{(r)}, K_s^{(r)})< c_2 \frac{s^r}{\log s}.$$\end{corollary}

Graph processes that iteratively add edges chosen uniformly at random subject to the
condition that some graph property is maintained have been used to generate interesting
combinatorial objects in a number of contexts.  In addition to the lower bound on
the Ramsey number $ r(K_3, K_s) $ given by the triangle-free graph process
(discussed above), the $H$-free graph process gives the best known lower bound
on the Ramsey number $ r(K_t,K_s) $ for $ t \ge 4$ fixed and the best known lower bound
on the Tur\'an numbers for some bipartite graphs \cite{BK}.  The process that forms
a subset of $ {\mathbb Z}_n$ by iteratively choosing elements to be members of the set
uniformly at random subject to the condition that the set does not contains a $k$-term
arithmetic progression produces a set that has interesting properties with respect to
the Gowers norm \cite{BB}.

The $T^{(r)}$-free (hyper)graph process can be viewed as an instance of the random greedy hypergraph
independent set process.  Let $ H$ be a hypergraph.  An independent
set in $ H $ is a set of vertices that contains no edge of $H$.  The random
greedy independent set process forms such a set by starting with an empty set of vertices and
iteratively choosing vertices uniformly
at random subject to the condition that the set of chosen vertices continues to be
an independent set.  We study the random greedy
independent set process for the hypergraph ${\mathcal H}_{T^{(r)}}$ which has vertex set $ \binom{[n]}{r} $
and edge set consisting of all copies of $ T^{(r)}$ on vertex set $[n]$. Note that, since an independent
 set in  ${\mathcal H}_{T^{(r)}}$ gives a $T^{(r)}$-free $r$-graph on point set $[n]$, the random
greedy independent set process on ${\mathcal H}_{T^{(r)}}$ is equivalent to the $T^{(r)}$-free process.
Our analysis of the
$T^{(r)}$-free process is based on recent work on the random greedy hypergraph independent set process
due to Bennett and Bohman \cite{BB}.

The remainder of the paper is organized as follows.  In the following Section we establish some
notation and recall the necessary facts from \cite{BB}.  The proof of Theorem~\ref{main} is given
in the Section that follows, modulo the proofs of some technical lemmas. These lemmas are proved
in the final Section by application of the differential equations method for proving dynamic
concentration.

\section{Preliminaries }

Let $\ca{H}$ be a hypergraph on vertex set $V=V(\ca{H})$.  For each set of vertices $A \subseteq V$, let $N_{\ca{H}}(A)$ denote the {\bf neighborhood} of $A$ in $\ca{H}$, the family of all sets $Y \subseteq V\setminus A$ for which $A \cup Y \in \ca{H}$.  We then define the {\bf degree} of $A$ in $\ca{H}$ to be $d_{\ca{H}}(A)=|N_{\ca{H}}(A)|$.  For a nonnegative integer $a$, we define $\Delta_a(\ca{H})$ to be the maximum of $d_{\ca{H}}(A)$ over all $A \in \bin{V}{a}$.  Next, for a pair of (not necessarily disjoint) sets $A,B \subseteq V$, we define the {\bf codegree} of $A$ and $B$ to be the number of sets $X \subseteq V \setminus (A \cup B)$ for which $A \cup X, B \cup X$ both lie in $\ca{H}$.

Recall that we define $G(i)$ to be the $r$-graph produced through $i$ steps of the $T^{(r)}$-free process.
We let $\ca{F}_i$ denote the natural filtration determined by the process (see \cite{B}, for example).  We also simplify our notation somewhat and write $N_i(A)$ in place of $N_{G(i)}(A)$, $d_i(A)$ in place of $d_{G(i)}(A)$, etc., when appropriate.

The $r$-graph $G(i)$ partitions $\bin{[n]}{r}$ into three sets $E(i),O(i),C(i)$.  The set $E(i)$ is simply the set of $i$ edges chosen in the first $i$ steps of the process.  The set $O(i)$ consists of the {\bf open} $r$-sets: all $e \in \bin{n}{r} \setminus E(i)$ for which $G(i)+e$ is $T^{(r)}$-free.  The $r$-sets in $C(i) := \bin{[n]}{r}\setminus (E(i) \cup O(i))$ are {\bf closed}.  Finally, for each open $r$-set $e \in O(i)$, we define the set $C_e(i)$ to consist of all open $r$-sets $f \in O(i)$ such that the graph $G(i)+e+f$ contains a copy of $T^{(r)}$ using both $e$ and $f$ as edges.  (That is, $C_e(i)$ consists of the open $r$-sets whose selection as the next edge $e_{i+1}$ would result in $e \in C(i+1)$.)

We now introduce some notation in preparation for our application of the results in \cite{BB}.
Set
\begin{align*}
N &:= \bin{n}{r} & D &:= (r+1)\cdot \bin{n-r}{r-1} & s &:= \frac{N}{D^{1/r}}.
\end{align*}
Note that $N$ is the size of the vertex set of the hypergraph ${\mathcal H}_{T^{(r)}}$ and $D$ is the vertex degree
of ${\mathcal H}_{T^{(r)}}$ (in other words, every $r$-set in $[n]$ is in $D$ copies of $ T^{(r)}$).  The
parameter $s$ is the `scaling' for the length of the process.  This choice is motivated by the heuristic that
$E(i)$ should be pseudorandom; that is, $E(i)$ should resemble in some ways a collection of $r$-sets chosen uniformly at random (without any further condition).  If this is indeed the case
then the probability that a given $r$-set is open would be roughly
\[  \left( 1 - \left(\frac{i}{N}\right)^r \right)^D \approx \exp \left\{ - \left(\frac{i}{N}\right)^r D \right\} \]
and a substantial number of $r$-sets are closed when roughly $s$ edges have been added.  In order to
discuss the evolution in more detail, we pass to a limit by introducing a
continuous time variable $t$ where $t=t(i) = i/s$.

The evolution of key parameters of the process closely
follow trajectories given by the functions
\[ q(t):= \exp\left\{ -t^r \right\} \; \; \;  \text{ and } \; \; \; c(t):= -q'(t) = rt^{r-1}q(t). \]
We introduce small constants $ \zeta, \gamma $ such that $ \zeta \ll \gamma \ll 1/r$.
(The notation $ \alpha \ll \beta $ here means that $ \alpha $ is chosen to be sufficiently small relative to $ \beta $.)
The point where we stop tracking the process is given by
\[\im :=\zeta \cdot ND^{-1/r}(\log^{1/r} N) \;\;\;\mbox{ and } \;\;\; \tm :=\im/s =  \zeta \log^{1/r} N.\]
For $i^* \ge 0$, let $\ca{T}_{i^*}$ denote the event that the following estimates hold for all steps $0 \le i \le i^*$:
\begin{equation}
\label{eq:open}
|O(i)| = \lp q(t) \pm N^{-\gamma} \rp N
\end{equation}
and for every open $r$-set $e \in O(i)$
\begin{equation}
\label{eq:degree}
|C_e(i)| = \lp c(t) \pm  N^{-\gamma}\rp D^{1/r}.
\end{equation}
It follows from the results of Bohman and Bennett that $ \ca{T}_{\im}$ holds
with high probability.
\begin{quote}
\begin{proof}
This follows from the estimates for the random greedy hypergraph independent set process given in \cite{BB} applied to the $(r+1)$-uniform hypergraph $ \ca{ H}_{T^{(r)}}$.  Verification of the conditions of Theorem 1.1 in \cite{BB} for this hypergraph is routine.  (Note that $ \Delta_\ell (\ca{H}_{T^{(r)}}) = \Theta( n^{r - \ell})$ and $ \Gamma_r( \ca{H}_{T^{(r)}}) = 0$.) The estimates (\ref{eq:open}) and (\ref{eq:degree}) above then follow from those on $|V(i)|$ and $d_2(v,i)$
given by (5) and (6) in \cite{BB}.
\end{proof}
\end{quote}

\noindent
Note that the fact that $ \ca{T}_{\im} $ holds with high probability does not
prove that the independence number of $ G(M)$ is $ O \lp ( n \log n)^{1/r} \rp$ with high probability.
This is proved below.

We will also make use of the following fact regarding
$r$-graphs that appear as subgraphs of the $T^{(r)}$-free process.
\begin{lemma}[\cite{BB} Lemma 4.2]
\label{lem:subgraph}
Fix a constant $L$ and suppose $e_1,\ldots,e_L \in \bin{[n]}{r}$ form a $T^{(r)}$-free hypergraph.  Then for all steps $j \le \im$,
\[\pr{\{e_1,\ldots,e_L\} \subseteq E(j)} = (j/N)^L\cdot (1 + o(1)).\]
\end{lemma}

We conclude this Section by noting that the desired bound on the independence number
of $G(M)$ can be viewed as a pseudorandom property of the $r$-graph $ G(i) $.  Indeed, if $ G(i)$
resembles a collection of $r$-sets chosen uniformly at random then the expected number of
independent sets of size $k$ would be
\[ \binom{n}{k} \left( 1 - \frac{ \binom{k}{r}}{ \binom{n}{r}} \right)^i
= \exp \left\{ \Theta \left( k \log n \right) - \Theta \left( i \frac{k^r}{n^r} \right) \right\}. \]
If the process lasts through $ i = \Theta( N D^{-1/r} ( \log^{1/r} N) ) = \Theta( n^{r -1 + 1/r} \log^{1/r} n) $
steps then we would anticipate an independence number of $ O\left( (n \log n)^{1/r} \right) $.  In the remainder of the paper
we make this heuristic calculation rigorous.

\section{Independence number: Proof of Theorem~\ref{main}}

We expand the list of constants given in the previous section by introducing
large constants $\kappa$ and $W$, and small constant $\epsilon $ such that
\begin{equation}
\label{eq:constants}
\frac{1}{\kappa} \ll \zeta \ll \frac{1}{W} \ll \ep \ll \gamma.
\end{equation}
In the course of the argument we introduce dynamic concentration phenomena
that will stated in terms of the error function
\[f(t) := \exp \left\{ W(t^{r} + t)\right\}.\]

Define the constant $\lambda:= \frac{\kappa - \gamma}{2}$, and then let
\begin{align*}
k &:= \kappa (n \log n)^{1/r} &  \mbox{ and } & & \ell &:= \lambda (n \log n)^{1/r},
\end{align*}
noting that as $\gamma$ is small, $k \approx 2\l$.
Our aim is to show that the independence number of $G(\im)$ is at most $k$ with high probability.  To do so, we will show that provided $\kappa$ is suitably large, w.h.p. for every step $0 \le i \le \im$, every $k$-element set of vertices has at least $ \Omega \lp q(t) \bin{k}{r} \rp$ open $r$-sets. As equation (\ref{eq:open}) establishes $ (1+o(1)) q(t) N$ open $r$-sets in total w.h.p., the probability that
$\ca{T}_{\im}$ holds and a given $k$-set remains independent over all $\im$ steps is then at most
\[ \prod_{i=1}^{\im}  \lp 1 -  \Omega \lp \frac{q(t)k^r}{q(t)N} \rp \rp = \lp 1 -  \Omega \lp \frac{\kappa^r \log n}{n^{r-1}} \rp \rp^{\im} = \exp\left\{-\zeta\kappa^r\cdot \Omega(n^{1/r} \log^{1+1/r} n) \right\},\]
where our $ O( \cdot), \Omega( \cdot), \Theta( \cdot ) $ notation does not
suppress any constant that appears in (\ref{eq:constants}).
Since $$n^k =\exp \left\{ \kappa \cdot O (n^{1/r} \log^{1+1/r} n) \right\},$$ this suffices by the union bound, provided $\kappa$ is suitably large
with respect to $r$ and $ \zeta $.

There is a significant obstacle to proving that every set of $k$ vertices contains the `right' number
of open $r$-sets.  Note that all $r$-sets within the neighborhood of an $(r-1)$-set are closed.
(To be precise, if $ A \in \binom{[n]}{r-1} $ then $ \binom{N_i(A)}{r} \subseteq C(i) $).  So a
set of $k$ vertices that has a large intersection with the neighborhood of an
$(r-1)$-set does not have the `right' number of open $r$-sets.
To overcome this obstacle, we extend the
argument in \cite{B} for bounding the independence number of the
triangle-free process.  Our argument has two steps:
\begin{enumerate}
    \item We apply the differential equations method for establishing dynamic concentration to show that
unless a certain `bad' condition occurs, a pair of disjoint $\l$-sets will have the `right' number of
open $r$-sets that are contained in the union of the pair of $\l$-sets and intersect both $\l$-sets, that is about $q(t)\cdot [\bin{2\l}{r}-2\bin{\l}{r}]$ open $r$-sets.
Note that $\bin{2\l}{r}-2\bin{\l}{r} > \frac{1}{3}\bin{k}{r}$, say, as $\gamma$ is small.
    \item We then argue that w.h.p., every $k$-set contains a (disjoint) pair of $\l$-sets which is `good', i.e., for which the bad condition does not occur.
\end{enumerate}
We formalize this with the notion of $r$-sets which are open `with respect to' a pair of disjoint $\l$-sets.

\begin{definition}
Fix a disjoint pair $A,B \in \bin{[n]}{\l}$.  The stopping time $\tau_{A,B}$ is the minimum of $\im$ and the first step $i$ for which there exists a $(r-1)$-set $X$ such that
$$ N_i(X)\cap A \ne \emptyset, \ \ \ \ N_i(X)\cap B \ne \emptyset, \ \ \ \text{ and } \ \ \ |N_i(X)\cap (A \cup B)| \ge k/n^{2\ep}.$$
\end{definition}

\begin{definition}
For each step $i \ge 0$, we say that an $r$-set $e \subseteq A \cup B$ is {\bf open with respect to the pair $A,B$ in $G(i)$} if $e \cap A \ne \emptyset$, $e \cap B \ne \emptyset$, and either

$\bullet$ $e \in O(i)$ {\em or}

$\bullet$ $e \in O(i-1) \cap C(i)$ and $i=\tau_{A,B}$.

\noindent Let $Q_{A,B}(i)$ count the number of $r$-sets which are open with respect to the pair $A,B$ in $G(i)$. \end{definition}

\begin{lemma}\label{lem:relopentraj}
With high probability, for every disjoint pair $A,B \in \bin{[n]}{\l}$ and all steps $0 \le i \le \tau_{A,B}$,
\begin{equation}\label{eq:relopenpairs}
Q_{A,B}(i) = \lp q(t) \pm \frac{f(t)}{n^{\ep}}\rp \cdot \left[\bin{2\l}{r}-2\bin{\l}{r}\right].
\end{equation}
\end{lemma}

\begin{lemma}\label{lem:largeopensets}
With high probability, for every step $0 \le i < \im$ and every set $K \in \bin{[n]}{k}$, there exists a pair of disjoint $\l$-sets $A,B$ contained in $K$ for which $\tau_{A,B} > i$.
\end{lemma}

\noindent
Lemmas \ref{lem:relopentraj} and \ref{lem:largeopensets}, respectively, complete steps
1 and 2 of the proof outlined above.  The `bad' condition for a pair $A,B$ of disjoint
$\ell$-sets is the event
that we have reached the stopping time $ \tau_{A,B}$; that is, the bad condition is that there
is some $ (r-1)$-set whose neighborhood intersects both $A$ and $B$ and has large intersection
with $A \cup B$.  Note that if $ i < \tau_{A,B} $ then $ Q_{A,B}$ is equal to the number of
open $r$-sets that are contained in $ A \cup B$ and intersect both $A$ and $B$.  Thus,
Lemma~\ref{lem:relopentraj} says that if we do not have the `bad' condition then we have the `right' number of
such sets.  Lemma~\ref{lem:largeopensets} then says that every $k$-set contains a pair disjoint pair $ A,B$ of
$ \ell$-sets for which the `bad' condition does not hold.  Taken together,
Lemmas~\ref{lem:relopentraj}~and~\ref{lem:largeopensets} yield that w.h.p., for every
step $0 \le i < \im$, every $k$-set contains at least
$q(t)(1 + o(1))[\bin{2\l}{r}-2\bin{\l}{r}] = \Omega \lp q(t) \binom{k}{r} \rp$ open $r$-sets, as required.
 We now prove Lemma~\ref{lem:largeopensets} modulo the proof of Lemma~\ref{lem:maxdeg} which bounds the
maximum degree of an $(r-1)$-set. Lemmas~\ref{lem:relopentraj}~and~\ref{lem:maxdeg} are proved in the
next Section.

\begin{proof}[Proof of Lemma~\ref{lem:largeopensets}]

We require a bound on the maximum degree of $(r-1)$-sets of vertices.  For each step $i \ge 0$ let $\ca{D}_{i}$ denote the event that $\Delta_{r-1}(G(i)) \le \ep (n \log n)^{1/(r-1)}$.
\begin{lemma}\label{lem:maxdeg}
$\ca{T}_{\im} \land \ca{D}_{\im}$ holds with high probability.
\end{lemma}
\noindent
The proof of Lemma~\ref{lem:maxdeg} is given in the next Section.

Fix a step $0 \le i < \im$, and a set $K \in \bin{[n]}{k}$.  Note that, by Lemma~\ref{lem:maxdeg},
we may assume that $\ca{D}_i$ holds.  We also note that the maximum co-degree of a pair of sets
$ A,B \in \binom{[n]}{r-1} $ is at most $ 5r$ with high probability.  This follows from Lemma~\ref{lem:subgraph}
and the union bound:
\begin{multline}
\label{eq:codegree}
\Pr \lp \exists A,B \in \binom{[n]}{r-1} \text{ with co-degree } 5r \rp \le \binom{n}{r-1} \binom{n}{r-1} n^{5r}  \left(\frac{i}{N}\right) ^{10r} \\ = n^{8-3r +o(1)}  = o(1).
\end{multline}
Given these two facts (i.e. these degree and co-degree bounds for $ (r-1)$-sets),
the remainder of the proof is deterministic.

To begin, define the set
\[\ca{X}:=\left\{X \in \bin{[n]}{r-1} : |N_i(X)\cap K| \ge k/n^{2\ep}\right\}.\]

\begin{claim}
$|\ca{X}| < 2 n^{2\ep}$.
\end{claim}
\begin{quote}
\begin{proof}
Suppose $\exists \ca{Y} \subseteq \ca{X}$ with $|\ca{Y}| = 2n^{2\ep}$. Let $N = \bigcup_{Y \in \ca{Y}} (N_i(Y) \cap K)$.  By inclusion-exclusion,
\[k \ge |N| \ge |\ca{Y}|\cdot (k /n^{2\ep}) - |\ca{Y}|^2 5r \ge 2k-20rn^{4\ep},\]
a contradiction as $\ep$ is small and $k = n^{1/r + o(1)}$.
\end{proof}
\end{quote}

Next, we `discard' from $K$ the vertices which are common neighbors of $(r-1)$-sets in $\ca{X}$: let
$$K_{bad}:=\{v \in K : \exists X,Y \in \ca{X} \mbox{ with } X \ne Y \mbox { and } v \in N_i(X)\cap N_i(Y)\}$$ and $K_{good}:=K\setminus K_{bad}$. Then
\[|K_{bad}| \le |\ca{X}|^2 5r\le 20rn^{4\ep} < \frac{\gamma}{2}\cdot (n \log n)^{1/r},\]
say, for large $n$.

We find disjoint $\l$-subsets $A,B$ of $K_{good}$ as follows, noting $|K_{good}| \ge 2\ell + (\gamma/2)(n \log n)^{1/r}$.  For each subset $\ca{Y}\subseteq \ca{X}$, let
\[N({\ca{Y}}) = \bigcup_{Y \in \ca{Y}} N_i(Y) \cap K_{good}.\]
Now, choose a maximal subset $\ca{X}^* \subseteq \ca{X}$ subject to $|N(\ca{X}^*)| \le \l$.  If $\ca{X}^*=\ca{X}$, then let $A,B$ be $\l$-sets satisfying $N(\ca{X}^*)\subseteq A \subseteq K_{good}$ and $B \subseteq K_{good}\setminus A$.

Otherwise, pick any set $X^*\in \ca{X}\setminus \ca{X}^*$, so
\[\ell < |N(\ca{X}^* \cup \{X^*\})| < \ell + \ep (n \log n)^{1/r};\]
let $A \subseteq N(\ca{X}^* \cup \{X^*\})$ and $B \subseteq K_{good}\setminus N(\ca{X}^* \cup \{X^*\})$ be $\ell$-sets.

Observe now that if $\ca{X}^*=\ca{X}$, then $N_i(X)\cap B=\emptyset$ for all $X \in \ca{X}$.  Otherwise, if $X \in \ca{X}^*\cup \{X^*\}$, $N_i(X)\cap B=\emptyset$, but if $X \in \ca{X}\setminus (\ca{X}^*\cup \{X^*\})$ then $N_i(X)\cap A=\emptyset$ as we are working within $K_{good}$.  In either case, for every $(r-1)$-set $X$ for which $|N_i(X) \cap (A \cup B)| \ge k/n^{2\ep}$ holds, either $N_i(X)\cap A = \emptyset$ or $N_i(X)\cap B = \emptyset$, and $\tau_{A,B} > i$ follows.

\end{proof}

\section{Dynamic Concentration}

In this section we prove Lemmas~\ref{lem:relopentraj}~and~\ref{lem:maxdeg}.  Both of
these statements assert dynamic concentration of key parameters of the $T^{(r)}$-free process.  We apply
the differential equations method for proving dynamic concentration, which we now briefly sketch.

Suppose we have a combinatorial stochastic process based on a ground set of size $n$ that generates
a natural filtration $ \ca{F}_0, \ca{F}_1, \dots $.  Suppose further that we
have a sequence of random variables $ A_0, A_1, \dots$ and that we would like to prove a dynamic concentration
statment of the form
\begin{equation}
\label{eq:dynamic}
A_i \le T_i + E_i  \ \ \text{ for all } \ \  0 \le i \le m(n) \ \   \text{ with high probability},
\end{equation}
where $ T_0, T_1, \dots $ is the expected trajectory of the sequence of random variables $A_i$
and $ E_0, E_1, \dots $ is a sequence of error functions.  (One is often interested in proving a lower
bound on $A_i$ in conjunction with (\ref{eq:dynamic}).  The argument for proving this is essentially
the same as the upper bound argument
that we discuss here.)   We often make this statement in the context
of a limit that we define in terms of a continuous time $t$ given by $ t = i/s$ where $ s$ is the
{\bf time scaling} of the process.  The limit of the expected trajectory is determined
by setting $ T_i = f(t) S(n) $ where $ S = S(n)$ is
the {\bf order scaling} of the random variable $A_i$.  Given these assumptions we should have
\[  \ev{A_{i+1} - A_i \mid  \ca{F}_i } = T_{i+1} - T_i =  \left[ f(t+1/s) - f(t) \right] S \approx f'(t) \cdot \frac{S}{s}.\]
Thus the trajectory is determined by the expected one-step change in $ A_i$.

We prove (\ref{eq:dynamic}) by applying facts regarding the probability of large deviations in
martingales with bounded differences.  In particular, we consider the sequence
\[  D_i = A_i - T_i - E_i. \]
Note that if we set $ T_0 = A_0$ (which is often the natural initial condition) then
$ D_0 = - E_0 $.  If we can establish that the sequence $ D_i$ is a supermartingale
and $ E_0$ is sufficiently large then it should be unlikely that $ D_i$ is ever positive,
and (\ref{eq:dynamic}) follows.  In order to complete such a proof we
show that the sequence $D_i$  is a supermartingale, a fact that is sometimes called the {\bf trend
hypothesis} (see Wormald \cite{W}).  The trend hypothesis will often impose a condition that the sequence of
error functions $ E_i$ is growing sufficiently quickly (i.e. the derivative of the limit of error
function is sufficiently large).  We then show that the one-step changes in $ D_i$ are bounded in
some way (this is sometimes called the {\bf boundedness hypothesis}). This puts us in the
position to apply a martingale inequality.  In order to get good bounds from the martingale inequality one
generally needs to make $ E_0$ large.

In this section we appeal to the following pair of martingale inequalities (see \cite{B}).
For positive reals $b,B$, the sequence $A_0,A_1,\ldots$ is said to be {\bf $(b,B)$-bounded} if $A_i - b \le A_{i+1} \le A_i + B$ for all $i \ge 0$.
\begin{lemma}\label{lem:submartinbd}
Suppose $b \le B/10$ and $0<a<bm$.  If $A_0,A_1,\ldots$ is a $(b,B)$-bounded submartingale, then $\pr{A_m \le A_0 - a} \le \exp \left\{ -a^2/3b m B \right\}$.
\end{lemma}
\begin{lemma}\label{lem:supmartinbd}
Suppose $b \le B/10$ and $0<a<bm$. If $A_0,A_1,\ldots$ is a $(b,B)$-bounded supermartingale, then $\pr{A_m \ge A_0 + a} \le \exp \left\{ -a^2/3b m B \right\}$.
\end{lemma}
\noindent
Our applications of these Lemmas make use of stopping times.  Formally speaking, a stopping time
is simply a postive integer-valued random variable $ \tau$ for
which $ \{ \tau \le n\} \in {\mathcal F}_n $.  In other words,
$\tau$ is a stopping time if the event $ \tau \le n$ is determined by the first $n$ steps of the
process.  We consider the stopped process $ ( D_{ i \wedge \tau }) $, where
$ x \wedge y := \min \{x,y\} $, in the place of the sequence $ D_0, D_1, \dots$.  Our stopping time $ \tau$
is the first step in the process when any condition on some short list of conditions
fails to hold, where the condition
$ D_i \le 0 $ is one of the conditions in the list.  Note that, since the variable $ ( D_{ i \wedge \tau }) $ does
not change once we reach the stopping time $ \tau$, we can assume that all conditions in the list
hold when we are proving the trend and boundedness hypotheses.  Also note that if the stopping time
$\tau'$ is simply the minimum of $ \im$ and the first step for which $ D_i >  0$ then

$ \{ D_{\im \wedge \tau'} > 0 \} $ contains the event $ \{ \exists i \le \im: D_i > 0 \}$.

\subsection{Proof of Lemma \ref{lem:maxdeg}.}

For each set $A \in \bin{[n]}{r-1}$ and step $i \ge 0$, let $O_A(i):=\{e \in O(i) : A \subseteq e\}$, and
$Q_A(i) = |O_A(i)|$.  We define sequences of random variables
\begin{align*}
Y^+_A(i) &:= q(t)\cdot n - Q_A(i) + f(t)\cdot n^{1-\ep},\\
Y^-_A(i) &:= q(t)\cdot n - Q_A(i) - f(t)\cdot n^{1-\ep},\\
Z_A(i) &:= d_i(A) - t\cdot D^{-1/r} n - f(t)q(t)^{-1}\cdot n^{1/r-\ep},
\end{align*}
Finally, we define the stopping time $\tau$ to be the minimum of $\bin{n}{r}$, the first step $i$ where $\ca{T}_i$ fails, or where any of $Y^+_A(i) < 0$, $Y^-_A(i) > 0$, or $Z_A(i) > 0$ holds for some $A \in \bin{[n]}{r-1}$.

To prove Lemma~\ref{lem:maxdeg}, we show that for each $A \in \bin{[n]}{r-1}$,
\begin{align}
\label{eq:YAplusbd} \pr{Y^+_A(\im \land \tau) < 0} &= o(n^{-(r-1)}),\\
\label{eq:YAminusbd} \pr{Y^-_A(\im \land \tau) > 0} &=o(n^{-(r-1)}), \mbox{ and }\\
\label{eq:ZAbd} \pr{Z_A(\im \land \tau) > 0} &=o(n^{-(r-1)}).
\end{align}
Consider the event
$ \tau \le \im $.  This event is the union of the event that $ \ca{ T}_{\im} $ fails and
the event that
there exists $ A \in \bin{[n]}{r-1}$ such that $ Y^+_A(\im \land \tau) < 0 $ or
$ Y^-_A(\im \land \tau) > 0 $ or $  Z_A(\im \land \tau) > 0 $.
Since $\ca{T}_{\im}$ holds with high probability, it follows from (\ref{eq:YAplusbd})--(\ref{eq:ZAbd}) and
the union bound that w.h.p. $\tau > \im$.  In particular, $Z_A(i) \le 0$ for
all $(r-1)$-sets $A$ and steps $0 \le i \le \im$.
It then follows -- since $\zeta \ll \min\{1/W,\ep \}$ implies that we may bound $f(\tm) < n^{\ep/2}$, say -- that we have
\[\Delta_{r-1}(G(\im)) \le \tm D^{-1/r}n + f(\tm) n^{1/r-\ep/2} = \zeta \cdot O( (n \log n)^{1/r}) \le \ep (n \log n)^{1/r},\]
for  $n$ sufficiently large.  (We remark in passing that
the bounds on $Y_A^{\pm}(i)$ given when $i<\tau$ are necessary for our proof of the bounds on $Z_A(i)$.)

For the remainder of this argument, fix a set $A \in \bin{[n]}{r-1}$.  We first prove \eqref{eq:YAplusbd} and \eqref{eq:YAminusbd}.
\begin{claim} \label{clm:YApmbds}
For $n$ sufficiently large, the variables $Y^+_A(0),\ldots,Y^+_A(\im \land \tau)$ form an $(O(n/s),O(n^{1 - \frac{1}{2r}}))$-bounded submartingale, and the variables $Y^-_A(0),\ldots,Y^-_A(\im \land \tau)$ form an $(O(n/s),O(n^{1 - \frac{1}{2r}}))$-bounded supermartingale.
\end{claim}

\begin{proof}
We begin by fixing a step $0 \le i \le \im$, and we assume that $i < \tau$.  Throughout we write $t=t(i)$, and note $t(i+1)=t+s^{-1}$ and that $s^{-1}=D^{1/r}/N=\Theta(n^{1-1/r-r})$.

To aid the calculations to follow, we begin by estimating the quantity $$\Xi := f(t+s^{-1}) - f(t).$$  Since $f(t)=\exp(Wt^r + Wt)$, $f'(t)$ and $f''(t)$ are products of $f(t)$ with polynomials in $t$.  As $\zeta \ll \max\{1/W,\ep\}$,  $\tm$ is polylogarithmic in $n$, and $n$ is large, we have the crude bounds $ f(t) \le n^{\epsilon/2} $ and
$f''(t) \le  n^{o(1)} f'(t) $.  Thus, by Taylor's Theorem,
\begin{equation}\label{eq:Xiestimate}
\left|\Xi - \frac{f'(t)}{s}\right| =  O\lp \frac{ \max_{ t^* \le \tm } f''(t^*) }{s^2}\rp= o\lp  \frac{f'(t)}{s}\rp.
\end{equation}

Observe now that we may write
\[Y_A^{\pm}(i+1)-Y_A^{\pm} (i) = (q(t+s^{-1})-q(t))\cdot n - \lp Q_A(i+1)-Q_A(i)\rp \pm \Xi\cdot  n^{1-\ep}.\]
(Note that this stands for the pair of equations in which each $\pm$ is replaced with $+$ or with $-$, respectively.)
We begin by establishing the boundedness claims: it is routine to verify that $c(t)$ and $c'(t)$
are bounded over the reals, implying
\begin{equation}\label{eq:qtdifferentialest}
|q(t+s^{-1})-q(t) - c(t)\cdot s^{-1}| = O(s^{-2}),
\end{equation}
and so
\[0 \ge \lp q(t + s^{-1})-q(t)\rp \cdot n \ge -O\lp \frac{n}{s} \rp.\]
As we have the bound $|f'(t)| = n^{\ep/2 + o(1) } $ and (\ref{eq:Xiestimate}), we have
\( |\Xi|\cdot n^{1-\ep}= o( n/s)\),
and the lower bound in the boundedness claims follows.  To establish the upper bounds,
it remains to bound $Q_A(i)-Q_A(i+1)$.  Consider the `next' edge $e_{i+1} \in O(i)$ and observe that $$Q_A(i)-Q_A(i+1) = |\lp\{e_{i+1}\} \cup C_{e_{i+1}}(i)\rp \cap O_A(i)|.$$  We bound $|C_{e_{i+1}}(i)\cap O_A(i)|$ by considering five cases depending on $|e_{i+1}\cap A|$:\\

\noindent {\bf Case 1: $|e_{i+1} \cap A| = 0$.}  Let $f \in O_A(i)\cap C_{e_{i+1}}(i)$: then $f = A \cup \{v\}$ for some vertex $v$, and since $G(i)+e_{i+1}+f$ contains a copy of $T^{(r)}$, $v \in e_{i+1}$ must hold. (Recall that every pair of edges in $T^{(r)}$ either shares exactly one or $r-1$ vertices.)  In this case, $|C_{e_{i+1}}(i)\cap O_A(i)| \le |e_{i+1}|=r$.\\

\noindent {\bf Case 2: $|e_{i+1} \cap A| = r-1$.} In this case, we may write $e_{i+1} = A \cup \{u_1\}$.  Now, let $f =A \cup \{v\} \in O_A(i)\cap C_{e_{i+1}}(i)$: since $f \cap e_{i+1}=A$ and $f \in C_{e_{i+1}}(i)$, there must exist vertices $u_2,\ldots,u_{r-1} \in N_i(A)$ so that $\{u_1,\ldots,u_{r-1},v\} \in E(i)$.  As then $v \in N_i(\{u_1,\ldots,u_{r-1}\})$, we may bound the number of such choices of $v$ (and hence of $f$) in this case above by $\Delta_{r-1}(G(i))^{r-1}\le \zeta^{r-1} (n \log n)^{(r-1)/r}$.  (Note the bound on the maximum degree follows as $Z_A(i) \le 0$ since $i < \tau$.)\\

\noindent {\bf Case 3:  $|e_{i+1} \cap A| = 1$.}  Write $A = \{x_1, \dots, x_{r-1}\}$, where we take $e_{i+1}\cap A = \{x_1\}$.  Let $f = A \cup \{v\} \in C_{e_{i+1}}(i) \cap O_A(i)$, and suppose $v \notin e_{i+1}$ (as there are at most $r-1$ such $v$), so $f \cap e_{i+1}=\{x_1\}$.  Consider a copy of $T^{(r)}$ in $G(i)+e_{i+1} + f$ using both $e_{i+1}$ and $f$ as edges: without loss of generality, we may assume that one of $e_{i+1},f$ maps to the edge $b_1$ of $T^{(r)}$, the other to the edge $a$.

If $e_{i+1}$ maps to $b_1$, then the $(r-1)$-set $e_{i+1}\setminus \{x_1\}$ maps to the common intersection $B$ of $b_1,\ldots,b_r$. Consequently $v \in N_i(e_{i+1}\setminus \{x_1\})$ must hold, and so there are at most $\Delta_{r-1}(G(i))$ such $r$-sets $f \in C_{e_{i+1}}(i) \cap O_A(i)$.

Otherwise, if $e_{i+1}$ maps to the edge $a$ and $f$ maps to $b_1$, then $\{x_2,\dots, x_{r-1},v\}$ maps to the common intersection $B$.  Thus, for each $u \in e_{i+1}\setminus \{x_1\}$ we have $\{u,x_2, \dots, x_{r-1},v\} \in E(i)$, implying $v \in N_i(\{u,x_2, \dots, x_{r-1}\})$ and (as $e_{i+1}$ is fixed), there are again at most $\Delta_{r-1}(G(i))$ such choices of $f$.  Thus, in this case we have $|C_{e_{i+1}}(i) \cap O_A(i)| \le 2 + 2\Delta_{r-1}(G(i)) = n^{1/r + o(1)}$.\\

\noindent {\bf Case 4: $1 < |e_{i+1} \cap A| = r-2$.}  Let $f = A \cup \{v\} \in O_A(i) \cap C_{e_{i+1}}(i)$.  Since $|f \cap e_{i+1}| \ge |A \cap e_{i+1}| > 1$, $|f \cap e_{i+1}| = r-1$ must hold, implying $v \in e_{i+1}$ and so $|O_A(i) \cap C_{e_{i+1}}(i)| \le r$ as in Case 1. \\

\noindent {\bf Case 5: $2 \le |e_{i+1} \cap A| \le r-3$.}  In this case, $|C_{e_{i+1}}(i)\cap O_A(i)| = 0$, as every $f \in
O_A(i)$ satisfies $1 \le |f \cap e_{i+1}| \le r-2$.\\[5mm]
\noindent From the cases above it follows that $Q_A(i)-Q_A(i+1) = n^{(r-1)/r + o(1)}$, and combining the above bounds, it follows that the sequences $Y^{\pm}_A(0),\ldots,Y^{\pm}_A(\im \land \tau)$ are $(O(n/s),O(n^{ 1 - \frac{1}{2r}}))$-bounded.

We turn now to the sub- and supermartingale claims: all expectation calculations to follow are implicitly conditioned on the history of the process up to step $i$, and we recall that we assume $i < \tau$.  For each open $r$-set $f \in O_A(i)$, we have $f \notin O_A(i+1)$ if and only if $e_{i+1} \in C_f(i) \cup \{f\}$.  Thus,
\begin{align*}
\ev{Y^{\pm}_A((i+1)) - Y^{\pm}_A(i)} &= (q(t+s^{-1})-q(t))\cdot n + \sum_{f \in O_A(i)} \frac{|C_f(i)|+1}{|O(i)|} \pm \Xi\cdot n^{1-\ep}.
\end{align*}

To establish the submartingale claim, consider the following chain of inequalities:
\begin{align*}
\sum_{f \in O_A(i)} \frac{|C_f(i)|+1}{|O(i)|}  &\ge (q(t)-f(t) n^{-\ep})\cdot n \cdot \frac{(c(t) - N^{-\gamma})\cdot D^{1/r} }{(q(t) + N^{-\gamma})\cdot N}\\
&= \lp 1 - \frac{N^{-\gamma} +f(t)n^{-\ep}}{q(t)+N^{-\gamma}}\rp(c(t)-N^{-\gamma})\cdot \frac{n}{s}.\\
&\ge \lp 1 - 2q(t)^{-1}f(t)n^{-\ep}\rp(c(t)-N^{-\gamma})\cdot \frac{n}{s}\\
&\ge \lp c(t) - 2c(t)q(t)^{-1}f(t)n^{-\ep}-N^{-\gamma}\rp\cdot \frac{n}{s}\\
&\ge \lp c(t) - (2c(t)q(t)^{-1}+1)\cdot f(t)n^{-\ep}\rp\cdot \frac{n}{s}.
\end{align*}
The first inequality follows from the bounds given by (\ref{eq:open}) and (\ref{eq:degree}) on the event $\ca{T}_i$ and as $Y_A^-(i) \le 0$, since $i<\tau$.  In the second and fourth inequalities we bounded $N^{-\gamma} < f(t)n^{-\ep}$, valid as $f(t) \ge 1$ and $\ep \ll \gamma$.  Thus, applying this bound and \eqref{eq:qtdifferentialest} gives
\begin{align*}
\ev{Y^+_A(i+1) - Y^+_A(i)} &\ge \Xi\cdot n^{1-\ep} - (2c(t)q(t)^{-1}+1)f(t)\frac{n^{1-\ep}}{s}-O\lp \frac{1}{s^2}\rp\\
&\ge \Xi\cdot n^{1-\ep} - (2c(t)q(t)^{-1}+2)f(t)\frac{n^{1-\ep}}{s}\\
&= \lp (1 + o(1))f'(t)- (2c(t)q(t)^{-1}+2)f(t)\rp \cdot\frac{n^{1-\ep}}{s}
\end{align*}
by \eqref{eq:Xiestimate}.  Since $f'(t)=(Wrt^{r-1}+W)f(t)$ and $2c(t)q(t)^{-1}= 2rt^{r-1}$, this final bound is nonnegative for large $n$ as $W$ is large, and so $Y_A^+(0),\ldots,Y_A^+(\im\land \tau)$ forms a submartingale.

We similarly bound $\ev{Q_A(i) - Q_A(i+1)}$ above to establish the supermartingale claim: as $1 < N^{-\gamma}D^{1/r}$ for large $n$, and as $\ca{T}_i$ holds and $Y_A^+(i) \ge 0$,
\begin{align*}
\sum_{f \in O_A(i)} \frac{|C_f(i)|+1}{|O(i)|}  &\le (q(t)+f(t) n^{-\ep})\cdot n \cdot \frac{(c(t) + 2N^{-\gamma})\cdot D^{1/r} }{(q(t) - N^{-\gamma})\cdot N}\\
&= \lp 1 + \frac{N^{-\gamma} +f(t)n^{-\ep}}{q(t)-N^{-\gamma}}\rp(c(t)+2N^{-\gamma})\cdot \frac{n}{s}\\
&\le \lp 1 + 4q(t)^{-1}f(t)n^{-\ep}\rp(c(t)+2N^{-\gamma})\cdot \frac{n}{s}\\
&\le \lp c(t) + (4c(t)q(t)^{-1}+4)f(t)n^{-\ep}\rp\cdot \frac{n}{s}.
\end{align*}
In addition to the bound $N^{-\gamma} \le f(t)n^{-\ep}$ used above, in the second inequality, we bounded $q(t)-N^{-\gamma} \ge q(t)/2$, and in the final we bounded $2N^{-\gamma}(1+4q(t)^{-1}f(t)n^{-\ep}) \le 4f(t)n^{-\ep}$ as $q(t)^{-1}f(t)n^{-\ep} \le 1$ which holds as $2W\zeta^r<\epsilon$ and $n$ is large.

Thus,
\begin{align*}
\ev{Y^-_A(i+1) - Y^-_A(i)} &\le -\Xi\cdot n^{1-\ep} + (4c(t)q(t)^{-1}+4)f(t)\frac{n^{1-\ep}}{s}+O\lp \frac{1}{s^2}\rp\\
&\le -\Xi\cdot n^{1-\ep} + (4c(t)q(t)^{-1}+5)f(t)\frac{n^{1-\ep}}{s}\\
&= \lp -(1 + o(1))f'(t)+ (4c(t)q(t)^{-1}+5)f(t)\rp \cdot\frac{n^{1-\ep}}{s},
\end{align*}
and again, as $W$ is large, this is strictly negative for $n$ sufficiently large.  Thus, the sequence $Y_A^-(0),\ldots,Y_A^-(\im\land \tau)$ forms a supermartingale, completing the proof.
\end{proof}

Since $Q_A(0)=n-r+1$,  $Y_A^+(0) = r-1 + n^{1-\ep}$ and $Y_A^-(0) = r-1-n^{1-\ep}$.  Applying Lemmas \ref{lem:submartinbd} and \ref{lem:supmartinbd}, respectively, we have
\begin{align*}
\pr{Y^+_A(\im \land \tau) < 0} &\le \exp\left\{ - \Omega \lp \frac{ n^{2-2\ep} }{ \frac{n}{s} \cdot \zeta s\log^{1/r} N \cdot n^{1 - \frac{1}{2r}})} \rp \right\}\\
&= \exp\left\{ - n^{\frac{1}{2r}-2\ep + o(1)}  \right\}\\
&< \exp\left\{ - n^{\frac{1}{4r}} \right\}
\end{align*}
(valid for large $n$ as $\ep$ is small), and an identical calculation yields
\begin{align*}
\pr{Y^-_A(\im \land \tau) > 0} &\le \exp\left\{ - n^{ \frac{1}{4r}} \right\}.
\end{align*}
We have established \eqref{eq:YAplusbd} and \eqref{eq:YAminusbd}.

It remains to prove \eqref{eq:ZAbd}.
\begin{claim} \label{clm:degbdmartingale}
The variables $Z_A(0),\ldots,Z_A(\im \land \tau)$ form a $(2n/N,2)$-bounded supermartingale.
\end{claim}

\begin{proof}

We begin by fixing a step $0 \le i \le \im$, and we assume that $i < \tau$.  Throughout we write $t=t(i)$.  Let $f_1(t)=f(t)q(t)^{-1} = \exp((W+1)t^{r} + Wt)$, and let $\Xi_1 := f_1(t+s^{-1})-f_1(t)$.
By the same reasoning given in Claim \ref{clm:YApmbds}, we may bound $|f_1(t)| < n^{\ep/2}$, say, for large $n$,
and $ f_1''(t) \le n^{o(1)} f_1'(t) $, and so
\begin{equation}\label{eq:Xi1estimate}
\left|\Xi_1 - \frac{f_1'(t)}{s}\right| =  O\lp \frac{\max_{t^* < \tm} f_1'' (t^*)}{s^2}\rp=o\lp  \frac{f_1'(t)}{s}\rp.
\end{equation}

Next, we observe that
\[Z_A(i+1)-Z_A(i) = d_{i+1}(A)-d_i(A) - \frac{n}{N} - \Xi_1 \cdot n^{1/r-\ep}.\]
The boundedness claim then follows for $n$ sufficiently large as $0 \le d_A(i+1)-d_A(i) \le 1$ and as
\[|\Xi_1|\cdot n^{1/r-\ep} \le n^{\ep/2 + o(1)}\cdot n^{1/r-\ep}\cdot s^{-1}  < n/N \]
as $s^{-1} =D^{1/r}/N = \Theta(n^{1-1/r}/N)$.

Turning to the supermartingale condition, observe that $d_{i+1}(A)=d_i(A)+1$ if and only if $e_{i+1}$ lies in the set of open $r$-sets counted by $Q_A(i)$.  Conditioned on the history of the process up to step $i$, it follows that
\begin{align}
\nonumber \ev{Z_A(i+1)-Z_A(i)} &= \frac{Q_A(i)}{|O(i)|} - \frac{n}{N} - \Xi_1\cdot n^{1/r-\ep} \\
\nonumber &\le \frac{(q(t) + f(t) n^{-\ep})\cdot n}{(q(t)-N^{-\gamma})\cdot N} - \frac{n}{N} - \Xi_1\cdot n^{1/r-\ep}\\
\nonumber &= \frac{N^{-\gamma} + f(t)n^{-\ep}}{(q(t)-N^{-\gamma})} \cdot\frac{n}{N} - \Xi_1\cdot n^{1/r-\ep}\\
\nonumber &\le (N^{-\gamma} + f(t) n^{-\ep})\cdot 2q(t)^{-1}\cdot \frac{n}{N}  - \Xi_1\cdot n^{1/r-\ep}\\
\nonumber &= (2q(t)^{-1}N^{-\gamma} + 2f_1(t) n^{-\ep})\cdot \frac{n}{N} - \Xi_1\cdot n^{1/r-\ep}\\
\label{eq:ZAlowerbd} &\le 4f_1(t) \cdot n^{-\ep}\cdot \frac{n}{N}  - \Xi_1\cdot n^{1/r-\ep}
\end{align}
Note that the first inequality holds as $\ca{T}_i$ and $Y_A^+(i) \ge 0$ since $i < \tau$, the second as $q(t)-N^{-\gamma} \ge q(t)/2$ since $\zeta \ll \gamma$, and the final as $N^{-\gamma} \le f(t)\cdot n^{-\ep}$, since $f(t) \ge 1$ and $\ep \ll \gamma$.  Noting that for large $n$, $D \ge n^{r-1}/r^r$ and so $s^{-1} \ge n^{1-1/r}/(rN)$, by \eqref{eq:Xi1estimate} we have
\begin{align*}
\Xi_1 \cdot n^{1/r-\ep} &= (1+o(1))\cdot \frac{f_1'(t)}{s}\cdot n^{1/r-\ep} \\
&\ge (1+o(1))\cdot \frac{Wf_1(t)\cdot n^{1-1/r}}{rN}n^{1/r-\ep}\\
&> \frac{W}{2r}\cdot f_1(t) \cdot n^{-\ep} \cdot \frac{n}{N}.
\end{align*}
Thus, since we assume $W$ is large, the supermartingale condition follows now from \eqref{eq:ZAlowerbd}.
\end{proof}

Finally, to show \eqref{eq:ZAbd}, we apply Lemma \ref{lem:supmartinbd} to yield
\begin{align*}
\pr{Z_A(\im \land \tau) > 0} &\le \exp \left\{  - \Omega \lp \frac{n^{2/r-2\ep}}{ \frac{n}{N} \cdot \zeta s\log^{1/r} N } \rp \right\} \\
&= \exp\left\{ - \frac{n^{2/r-2\ep}}{ n^{1 - (r-1)/r + o(1)} }  \right\} \\
&= \exp\left\{ - n^{1/r - 2\ep- o(1)} \right\}
\end{align*}
which suffices as $\ep$ is small.  This completes the proof of Lemma \ref{lem:maxdeg}.

\subsection{Proof of Lemma \ref{lem:relopentraj}} We begin by letting
\[S = S(n) = \bin{2\ell}{r} - 2\bin{\ell}{r},\]
and we note that $S = \Theta(k^r)$.

We fix a pair $A,B$ of disjoint $\ell$-element subsets of $[n]$, and define the following sequences of random variables:  for each step $i \ge 0$, let
\begin{align*}
X^+(i) &= q(t)\cdot S - Q_{A,B}(i) + f(t)\cdot Sn^{-\ep}, \mbox{ and } \\
X^-(i) &= q(t)\cdot S - Q_{A,B}(i) - f(t)\cdot Sn^{-\ep}.
\end{align*}
We next define the stopping time $\tau^*$ to be the minimum of $\tau_{A,B}$ and the first step $i$ for which $X^+(i) \le 0$, $X^-(i) \ge 0$, or the event $\ca{T}_i$ fails to hold.

\begin{claim}\label{clm:lem1mart}
The sequence $X^+(0),\ldots,X^+(\im \land \tau^*)$ forms a $(O(k^r/s),O(k^{r-1}/n^{4\ep}))$-bounded submartingale, and the sequence $X^-(0),\ldots,X^-(\im \land \tau^*)$ forms a $(O(k^r
/s),O(k^{r-1}/n^{4\ep}))$-bounded supermartingale.
\end{claim}

\begin{proof}
We fix a step $0 \le i \le \im$, and we suppose that $i < \tau^*$. Throughout we write $t=t(i)$, and note $t(i+1)=t+s^{-1}$ and that $s^{-1}=D^{1/r}/N=\Theta(n^{1-1/r-r})$.

To aid the calculations to follow, we begin by estimating the quantity $\Xi := f(t+s^{-1}) - f(t)$.  Recall equation
(\ref{eq:Xiestimate}): 
\begin{equation*}
\left|\Xi - \frac{f'(t)}{s}\right| =  O\lp \frac{ \max_{ t^* \le \tm } f''(t^*) }{s^2}\rp= o\lp  \frac{f'(t)}{s}\rp.
\end{equation*}
Observe that we may write
\[X^{\pm}(i+1)-X^{\pm} (i) = (q(t+s^{-1})-q(t))\cdot S - \lp Q_{A,B}(i+1)-Q_{A,B}(i)\rp \pm \Xi\cdot  Sn^{-\ep}.\]
(As above, this stands for the pair of equations in which each $\pm$ is replaced with $+$ or with $-$, respectively.)
We begin by establishing the boundedness claims:
by (\ref{eq:qtdifferentialest}) and
as $S = \Theta(k^r)$, we have
\[0 \ge \lp q(t + s^{-1})-q(t)\rp \cdot S \ge -O\lp \frac{k^r}{s} \rp.\]
Next, bounding $|f'(t)| \le n^{\ep/2 + o(1)}$,
\[|\Xi|\cdot Sn^{-\ep} \le n^{-\ep/2+o(1)}\cdot \frac{k^r}{s} \]
In order to establish the boundedness part of the claim,
it remains to bound the quantity $Q_{A,B}(i+1)-Q_{A,B}(i)$.
Let $O_{A,B}(i)$ denote the set of $r$-sets that are
open with respect to the pair $A,B$ in $G(i)$, and let $O_{\tau}$ denote the set of
all open $r$-sets whose selection as $e_{i+1}$ would result in $\tau_{A,B}=i+1$.

Now, if $e_{i+1} \in O_{\tau}$, then $Q_{A,B}(i+1)-Q_{A,B}(i) = 0$ by definition, and, otherwise, we have
\[Q_{A,B}(i+1)-Q_{A,B}(i) = -|O_{A,B}(i) \cap ( C_{e_{i+1}}(i) \cup \{e_{i+1}\})|.\]
It suffices, then, to bound the quantity $|C_e(i) \cap O_{A,B}(i)|$ for all $e \in O(i)\setminus O_{\tau}$: fix such an open $r$-set $e$.  Now, for any $f \in C_e(i) \cap O_{A,B}(i)$, there is a copy $T_{r,f}$ of $T^{(r)}$ in the graph $G(i)+e+f$ using both $e$ and $f$ as edges.  Up to isomorphism, there are only three possibilities for the pair $(e,f)$ in that copy: $(e,f)$ maps to $(b_1,b_2)$, or to $(b_1,a)$, or to $(a,b_1)$. We treat these three cases separately.\\

\noindent {\bf Case 1: $(e,f)$ maps to $(b_1,b_2)$.} In this case, the $r-1$ vertices that map to the set $R$ lie entirely in $e$, and $f$ is the union of those $r-1$ vertices along with another vertex lying in $A \cup B$.  Thus, we may bound the total number of such $f$ above by $rk$.\\

\noindent {\bf Case 2: $(e,f)$ maps to $(b_1,a)$.}  Let $R' = e-f$, the set of $r-1$ vertices shared by all edges $b_j$ in this copy of $T^{(r)}$. Then $f-e \subseteq N_i(R')$: since $f \cap A \ne \emptyset$ and $f \cap B \ne \emptyset$ (as $f \in O_{A,B}(i)$), and since $e \notin O_{\tau}$, it follows that $|N_i(R') \cap( A \cup B)| \le k/n^{2\ep}$.  Thus, for a fixed such choice of $R'$ there are fewer than $(k/n^{2\ep})^{r-1}$ such open $r$-sets $f$, yielding a total bound of at most $r(k/n^{2\ep})^{r-1}$.\\

\noindent {\bf Case 3: $(e,f)$ maps to $(a,b_1)$.}  There exists an $(r-1)$-set $R' \subseteq A \cup B$ and a vertex $v \in e$ so that $f = R' \cup \{v\}$ and so that $e\setminus \{v\} \subseteq N_i(R')$.  To bound the number of such $f$, it suffices to bound the number of $(r-1)$-sets $R' \subseteq A \cup B$ for which $N_i(R')$ contains $(r-1)$ vertices from $e$.

To that end, fix a vertex $v \in e$ and let $\ca{H}_v$ denote the $(r-1)$-uniform hypergraph on $(A \cup B) \setminus e $ whose edges are the $(r-1)$-subsets $X$ for which $N_i(X) \supseteq e\setminus \{v\}$.  We claim that
\[\Delta_{r-2}(\ca{H}_v) < 4r.\]
Suppose to the contrary that this does not hold: then there exist an $(r-2)$-set $Y \subseteq (A \cup B)\setminus e$ and vertices $x_1,x_2,\ldots,x_{4r}\in (A\cup B)\setminus (Y \cup e )$ so that for each for each vertex $u \in e\setminus \{v\}$, $\{u\} \cup Y \cup \{x_j\} \in E(i)$ for $1 \le j \le 4r$.  It follows from Lemma~\ref{lem:subgraph} that such a
configuration does not appear in $ G(i)$.  Indeed, as this configuration spans $ 6r-3 $ vertices and has $ 4r(r-1) $ edges, the
probability that such a configuration appears is at most
\[ n^{6r-3} \left( \frac{i}{N} \right)^{4r(r-1)} = n^{6r-3 - 4 (r-1)^2 + o(1)} = o(1). \]
It follows that $|\ca{H}_v| < 4r \bin{k}{r-2}$, and thus the total number of such open $r$-sets $f$ as above is less than $4r^2k^{r-2}$.\\

As $\ep$ is small and as $k = n^{1/r + o(1)}$, it follows that for large $n$ we have
\[ |C_e(i) \cap O_{A,B}(i)| \le rk + r\cdot (k/n^{2\ep})^{r-1} + 4r^2k^{r-2} = O(k^{r-1}/n^{2\ep(r-1)}),\]
and as $r \ge 3$ we conclude that
\[0 \ge Q_{A,B}(i+1)-Q_{A,B}(i)  = -O(k^{r-1}/n^{4\ep}).\]
Thus, it follows that the sequences $X^{\pm}(0), \ldots, X^{\pm}(\im \land \tau^*)$ are $(O(k^r/s), O(k^{r-1}/n^{4\ep}))$-bounded as claimed.\\

We now turn to the sub- and supermartingale claims, and we remark that all expectation and probability calculations to follow are implicitly conditioned on the history of the process up to step $i$.  We begin by bounding the expected value of $Q_{A,B}(i+1)-Q_{A,B}(i)$.  Recall that we assume $i < \tau_{A,B}$ and that $O_{\tau} \subseteq O(i)$ consists of the open $r$-sets whose selection as $e_{i+1}$ would yield $\tau_{A,B} = i+1$. We claim that
\begin{equation}\label{eq:otauest}
|O_{\tau}| \le 4n^{2\ep}\cdot k\,
\end{equation}
To see this, let
\[\ca{R}:=\left\{X \in \bin{[n]}{r-1} : |N_i(X)\cap (A \cup B)| \ge k/(2n^{2\ep})\right\}.\]
Then $|\ca{R}| < 4n^{2\ep}$, which can be argued as follows. Suppose by way of contradiction that $\exists \ca{S} \subseteq \ca{R}$ with $|\ca{S}| = 4n^{2\ep}$. Let $N = \bigcup_{Y \in \ca{S}} (N_i(Y) \cap (A \cup B))$. By inclusion-exclusion and the fact that Lemma~\ref{lem:subgraph} implies that the co-degree of any pair of $(r-1)$-sets is at most $5r$ (see (\ref{eq:codegree})), we have
\[k \ge |N| \ge |\ca{S}|\cdot k /(2n^{2\ep}) - |\ca{S}|^2 5r \ge 2k-80rn^{4\ep},\]
a contradiction as $\ep$ is small and $k = n^{1/r + o(1)}$.  To deduce \eqref{eq:otauest} it suffices to observe that each open $r$-set $e \in O_{\tau}$ can be written $e = \{v\} \cup X$ for some vertex $v \in A \cup B$ and $(r-1)$-set $X$ satisfying $|N_i(X) \cap (A \cup B)|\ge k/n^{2\ep}-1$ (and thus $X \in \ca{R}$).

Conditioning on the event $e_{i+1} \notin O_{\tau}$ then yields
\begin{align*}
\ev{Q_{A,B}(i+1)-Q_{A,B}(i)}
&= -\sum_{e \in O_{A,B}(i)} \frac{|C_e(i)\setminus O_{\tau}|}{|O(i)|}
\end{align*} by linearity of expectation. Consequently,
\[\ev{X^{\pm}(i+1)-X^{\pm}(i)} = (q(t+s^{-1})-q(t))\cdot S + \sum_{e \in O_{A,B}(i)} \frac{|C_e(i)\setminus O_{\tau}|}{|O(i)|} \pm \Xi\cdot Sn^{-\ep}.\]

To establish the submartingale claim, we note first that as $r \ge 3$ and $\ep \ll \gamma \ll 1/r$, from \eqref{eq:otauest} we have $|O_{\tau}| = n^{1/r + 2\ep + o(1)} < N^{-\gamma}\cdot D^{1/r}$.
Now, as $i < \tau^*$,  $\ca{T}_i$ and $X^-(i) \le 0$ hold, we have
\begin{align*}
\sum_{e \in O_{A,B}(i)} \frac{|C_e(i)\setminus O_{\tau}|}{|O(i)|} &\ge \lp q(t)-\frac{f(t)}{n^{\ep}}\rp \cdot S \cdot \frac{(c(t)-2N^{-\gamma})D^{1/r}}{(q(t)+N^{-\gamma})N}\\
&= \lp 1 - \frac{N^{-\gamma} +f(t)n^{-\ep}}{q(t)+N^{-\gamma}}\rp(c(t)-2N^{-\gamma})\cdot \frac{S}{s}\\
&\ge \lp 1 - 2q(t)^{-1}f(t)n^{-\ep}\rp(c(t)-2N^{-\gamma})\cdot \frac{S}{s}\\
&\ge \lp c(t) - 2c(t)q(t)^{-1}f(t)n^{-\ep}-2N^{-\gamma}\rp\cdot \frac{S}{s}\\
&\ge \lp c(t) - (2c(t)q(t)^{-1}+1)f(t)n^{-\ep}\rp\cdot \frac{S}{s}.
\end{align*}
Note that these bounds follow for large $n$ since $f(t) \ge 1$ and $\ep \ll \gamma$ imply $N^{-\gamma} \le f(t)n^{-\ep}/2$. Applying this and \eqref{eq:qtdifferentialest} gives
\begin{align*}
\ev{X^+(i+1) - X^+(i)} &\ge \Xi\cdot Sn^{-\ep} - (2c(t)q(t)^{-1}+1)f(t)\frac{Sn^{-\ep}}{s}-O\lp \frac{1}{s^2}\rp\\
&\ge \Xi\cdot Sn^{-\ep} - (2c(t)q(t)^{-1}+2)f(t)\frac{Sn^{-\ep}}{s}\\
&= \lp (1 + o(1))f'(t)- (2c(t)q(t)^{-1}+2)f(t)\rp \cdot\frac{Sn^{-\ep}}{s}
\end{align*}
by \eqref{eq:Xiestimate}.  Since $f'(t)=(Wrt^{r-1}+W)f(t)$ and $2c(t)q(t)^{-1}= 2rt^{r-1}$, this final bound is nonnegative for large $n$ as $W$ is large, and so $X^+(0),\ldots,X^+(\im\land \tau)$ forms a submartingale.

Turning to the supermartingale claim, we take a similar approach and begin by noting as $\ca{T}_i$ holds and $X^+(i) \ge 0$,
\begin{align*}
\sum_{e \in O_{A,B}(i)} \frac{|C_e(i)\setminus O_{\tau}|}{|O(i)|} &\le \lp q(t)+\frac{f(t)}{n^{\ep}}\rp \cdot S \cdot \frac{(c(t)+N^{-\gamma})D^{1/r}}{(q(t)-N^{-\gamma})N}\\
&= \lp 1 + \frac{N^{-\gamma} +f(t)n^{-\ep}}{q(t)-N^{-\gamma}}\rp(c(t)+N^{-\gamma})\cdot \frac{S}{s}\\
&\le \lp 1 + 2q(t)^{-1}f(t)n^{-\ep} \rp(c(t)+N^{-\gamma})\cdot \frac{S}{s}\\
&\le \lp c(t) + (2c(t)q(t)^{-1}+1)f(t)n^{-\ep} \rp\cdot \frac{S}{s}.
\end{align*}
The supermartingale condition then follows in essentially the same way as the submartingale condition above.
\end{proof}

Now, as $X^+(0) = Sn^{-\ep}$, $X^-(0) = -Sn^{-\ep}$, $S = \Theta(k^r)$ and $\im = s\cdot n^{o(1)}$, it follows from Claim \ref{clm:lem1mart} and Lemmas \ref{lem:submartinbd} and \ref{lem:supmartinbd} that
\[\pr{X^+(\im \land \tau^*) \le 0} \le  \exp \left\{ - \Omega \lp \frac{ S^2 n^{-2 \ep}}{ \frac{ k^r}{s} \cdot \frac{ k^{r-1}}{n^{4 \ep}} \cdot s n^{o(1)}} \rp \right\} =    \exp\left\{ - k \cdot n^{2\ep - o(1)} \right\}. \]
Simillarly, we have
\[\pr{X^-(\im \land \tau^*) \ge 0} \le \exp \left\{ - k \cdot n^{2\ep - o(1)} \right\} .\]
Since there are fewer than $n^{2k} = \exp \{ 2k \log n \}$ choices of the pair of sets $A$ and $B$, Lemma \ref{lem:relopentraj} follows from the union bound.

\bibliographystyle{plain}

\begin{thebibliography}{99}
\setlength{\itemsep}{.5mm}


\bibitem{AKS} M. Ajtai, J. Koml\'{o}s and E. Szemer\'{e}di:
         A Note on Ramsey Numbers. \emph{J. Comb. Theory Ser. A}, {\bf 29}
         (1980) 354--360.


\bibitem{BB} P. Bennett and T. Bohman, A note on the random greedy independent set algorithm, submitted,
\href{http://arxiv.org/abs/1308.3732}{\tt arXiv.1308.3732}.

\bibitem{B} T. Bohman, The triangle-free process, {\em Advances in Mathematics}, {\bf 221} (2009)  1653--1677.

\bibitem{BFM}  T. Bohman, A. Frieze, D. Mubayi, Coloring $H$-free hypergraphs, {\em Random Structures and Algorithms}, {\bf 36} (2010) 11--25.

\bibitem{BK} T. Bohman and P. Keevash, The early evolution of the $H$-free process, {\em Inventiones Mathematicae}, {\bf 181} (2010) 291--336.

\bibitem{BK2}  T. Bohman and P. Keevash, Dynamic concentration of the triangle-free process,
submitted, \href{http://arxiv.org/abs/1302.5963}{\tt arXiv.1302.5963}.

\bibitem{CM} J. Cooper, D. Mubayi, Coloring sparse hypergraphs, submitted, \href{http://arxiv.org/abs/1404.2895}{\tt arXiv.1404.2895}

\bibitem{FGM}
G. Fiz Pontiveros, S. Griffiths, R. Morris, The triangle-free process and $R(3,k)$, submitted, \href{http://arxiv.org/abs/1302.6279}
{\tt arXiv.1302.6279}.

\bibitem{K} J.H. Kim, The Ramsey number $R(3,t)$ has order of magnitude $t^2/\log t$, {\em Random Structures \& Algorithms}, {\bf 7} (1995) 173--207.

\bibitem{KMV} A. Kostochka, D. Mubayi, J. Verstra\"ete, On independent sets in hypergraphs, {\em Random Structures \& Algorithms}, {\bf 44} 224--239.

\bibitem{PR}
K. T. Phelps, V. R{\"o}dl, Steiner triple systems with minimum independence number, {\em Ars Combinatoria}, {\bf 21}
(1986) 167--172


\bibitem{S83} J.B. Shearer: A note on the independence number of
         triangle-free graphs. {\em Discrete Math.} {\bf 46} (1983) 83--87.

\bibitem{S91} J.B. Shearer: A note on the independence number of
         triangle-free graphs II. {\em J. Combintorial Theory Series B},
         {\bf 2} 300--307.

\bibitem{W} N. Wormald, The differential equation method for random graph processes and greedy algorithms, in {\em Lectures on Approximation and Randomized Algorithms,} M. Karonski and H.J. Pr\"omel, editors, 1999, pp. 73-155.


\end{thebibliography}

\end{document}